\documentclass{amsart}

\usepackage[utf8]{inputenc}
\usepackage{comment}
\usepackage{amsmath}
\usepackage{amssymb}
\usepackage{amsthm}
\usepackage[all]{xy}

\usepackage{tikz}
\usetikzlibrary{cd,decorations.pathreplacing,decorations.markings}
\tikzset{mid arrow/.style={postaction={decorate,decoration={
          markings,
          mark=at position .5 with {\arrow[#1]{stealth}}
      }}}}

\swapnumbers
\newtheorem{theorem}{Theorem}[section]
\newtheorem*{utheorem}{Theorem}
\newtheorem*{ucorollary}{Corollary}
\newtheorem{definition}[theorem]{Definition}
\newtheorem{proposition}[theorem]{Proposition}
\newtheorem{corollary}[theorem]{Corollary}
\newtheorem{lemma}[theorem]{Lemma}
\theoremstyle{definition}
\newtheorem*{uremark}{Remark}
\newtheorem{remark}[theorem]{Remark}
\newtheorem{example}[theorem]{Example}

\newcommand*\cocolon{%
        \nobreak
        \mskip6mu plus1mu
        \mathpunct{}%
        \nonscript
        \mkern-\thinmuskip
        {:}%
        \mskip2mu
        \relax
}

\usepackage{pict2e,picture}
\makeatletter
\DeclareRobustCommand{\bbDelta}{{\mathpalette\bb@Delta\relax}}
\newcommand{\bb@Delta}[2]{%
  \begingroup
  \sbox\z@{$\m@th#1\Delta$}%
  \dimendef\Dht=6 \dimendef\Dwd=8
  \setlength{\Dwd}{\wd\z@}%
  \setlength{\Dht}{\ht\z@}%
  \begin{picture}(\Dwd,\Dht)
  \put(0,0){$\m@th#1\Delta$}
  \put(.52\Dwd,.77\Dht){\line(-13,-26){.35\Dht}}
  \end{picture}%
  \endgroup
}
\makeatother

\newcommand{\cC}{\mathcal{C}}
\newcommand{\C}{\mathrm{C}}
\newcommand{\J}{\mathrm{J}}
\newcommand{\Cb}{\mathrm{Q}}

\newcommand{\fC}{\mathfrak{C}}
\newcommand{\Set}{\mathrm{Set}}
\newcommand{\sSet}{\mathrm{sSet}}
\newcommand{\An}{\mathrm{An}}
\newcommand{\ssSet}{\mathrm{ssSet}}
\newcommand{\di}{\mathrm{dir}}

\newcommand{\Nc}{\mathfrak{N}}
\newcommand{\Sg}{\mathrm{W}}
\newcommand{\QL}{\mathrm{Q}^{\mathrm L}}
\newcommand{\QR}{\mathrm{Q}^{\mathrm R}}

\DeclareMathOperator{\dec}{dec}
\DeclareMathOperator{\Sing}{Sing}
\DeclareMathOperator{\Hom}{Hom}
\newcommand{\HomL}{\Hom^{\mathrm L}}
\newcommand{\HomR}{\Hom^{\mathrm R}}

\DeclareMathOperator{\Cat}{Cat}
\DeclareMathOperator{\CatKan}{\Cat^{\mathrm{Kan}}}
\DeclareMathOperator{\CatPan}{\Cat^{\pi\An}}
\DeclareMathOperator{\Cut}{Cut}
\DeclareMathOperator{\Fun}{Fun}

\DeclareMathOperator{\Coeq}{Coeq}

\newcommand{\tworightarrows}{\rightrightarrows} 

\title{Mapping spaces in homotopy coherent nerves}

\author[F.~Hebestreit]{Fabian Hebestreit}
\address{Mathematisches Institut, RFWU Bonn, Germany}
\email{f.hebestreit@math.uni-bonn.de}

\author[A.~Krause]{Achim Krause}
\address{Mathematisches Institut, WWU M\"unster, Germany}
\email{krauseac@uni-muenster.de}

\begin{document}

\begin{abstract}
We give a direct proof that middle mapping spaces in coherent nerves of Kan enriched categories have the same homotopy type as the original mapping spaces.
\end{abstract}
\maketitle
\tableofcontents
\section{Introduction}

For a quasi-category $\cC$, there are three common definitions of morphism complexes, often termed left, right and middle. For morphisms from $x$ to $y$ they can be defined as the fibres of the canonical maps
\[\cC_{/y} \rightarrow \cC, \quad \cC_{x/} \rightarrow \cC \quad \text{and} \quad \mathrm{Ar}(\cC) \rightarrow \cC \times \cC\]
over $x, y$ and $(x,y)$, respectively, and we shall denote them by $\HomL_\cC$, $\HomR_\cC$ and $\Hom_\cC$.

There are canonical inclusions
\[\HomL_\cC(x,y) \longrightarrow \Hom_\cC(x,y) \longleftarrow \HomR_\cC(x,y)\]
that are homotopy equivalences of Kan complexes (though not usually isomorphisms unless $\cC$ is an ordinary category) by non-trivial results of Joyal. Though all of these have their uses, arguably it is the middle one which is the most natural: For example, its definition is symmetric in source and target, it readily compares to morphism spaces of complete Segal spaces, has the most clearly defined composition law and features in the standard proof that fully faithful and essentially surjective functors are equivalences.

Now, one of the most prominent constructions of quasi-categories is the coherent nerve $\Nc$ of Cordier and Porter, which takes as input a category enriched in Kan complexes. There arises the problem of relating, for such a category $\C$ the given morphism complexes $\C(x,y)$ with those of $\Nc(\C)$. This was done by Lurie by producing an explicit equivalence $\C(x,y) \rightarrow \HomL_{\Nc(\C)}(x,y)$, which can be combined with Joyal's result above to obtain an equivalence $\sigma \colon \C(x,y) \rightarrow \Hom_{\Nc(\C)}(x,y)$. The main goal of the present note is to present a direct proof of this latter result avoiding the rather less elementary theorem of Joyal.

\begin{utheorem}
The map 
\[
\sigma \colon \C(x,y) \longrightarrow \Hom_{\Nc(\C)}(x,y)
\]
is a homotopy equivalence of Kan complexes for every Kan enriched category $\C$ with objects $x$ and $y$.
\end{utheorem}

By formal considerations there is in fact a tautological isomorphism \[\Hom_{\Nc(\C)}(x,y) \cong \Sing_\Sg(\C(x,y))\] for some particular $\Sg \colon \bbDelta \rightarrow \sSet$; here and in general we denote for a functor $T \colon \mathrm A \rightarrow \mathrm B$ by 
\[|\cdot|_T\colon \mathcal P(\mathrm A) \longleftrightarrow \mathrm B \cocolon \Sing_T\]
the adjunction arising by colimit extension of $T$ (assuming of course that $A$ is small and $B$ cocomplete).

The map $\sigma$ of the theorem is then induced by an explicit transformation $\sigma \colon \Sg \rightarrow \Delta$, where $\Delta$ denotes the cosimplicial object of $\sSet$ formed by the simplices (which satisfies $|\cdot|_\Delta = \mathrm{id}_{\sSet} = \Sing_\Delta$). Our proof of the theorem boils down to the fact that the cosimplicial object $\Sg$ is both Reedy cofibrant and termwise contractible. This is the exact same strategy employed in the analogous result for left mapping spaces, except that the relevant cosimplicial object is a lot less complicated in that case. We, in fact, verify both cases simultaneously with comparatively little explicit combinatorial manipulation. 

We also provide a somewhat more detailed analysis of mapping spaces in coherent nerves, in particular producing (to us) surprising isomorphisms
\[\HomL_{\Nc(\C)}(x,y) \cong \HomR_{\Nc(\C)}(x,y)^\mathrm{op} \quad \text{and} \quad \Hom_{\Nc(\C)}(x,y) \cong \Hom_{\Nc(\C)}(x,y)^\mathrm{op}.\]

\begin{uremark}
Our hope is that the proof we present here can be used in a first course on quasi-categories, which in our experience is better served by considering middle mapping spaces, than either the left or right variants. It was precisely this context which originally provided the motivation for the present note.

Our proofs are therefore designed to work with minimal input (though we are not aware that they could be simplified much by investing results other than the comparison of left and middle mapping spaces). We will, however, either have to rely on the non-trivial fact that cofibrations of simplicial sets that are also weak homotopy equivalences are in fact anodyne extensions or on the consequence of Joyal's lifting theorem that middle mapping spaces in quasi-categories are Kan; this dependence could be avoided by directly checking that $|\cdot|_\Sg$ preserves anodyne extensions, but we refrain from doing so here, since we would expect at least one of these results to be available by the time the statements we prove here become relevant in any context.
\end{uremark}

\subsection*{Acknowledgements}
We want to heartily thank Thomas Nikolaus and Christoph Winges for several very useful discussions, and the participants of the lecture course `Introduction to higher categories' in the winter term 19/20 at the University of Bonn for enduring the first exposition of the material we present here.

FH is a member of the Hausdorff Center for Mathematics at the University of Bonn, and AK of the Cluster `Mathematics M\"unster: Dynamics-Geometry-Structure', both funded by the German Research Foundation (DFG) under grant nos. EXC 2047, project-ID 390685813 and EXC 2044, project-ID 390685587, respectively. AK was further funded by the DFG through the collaborative research centre `Geometry: Deformations an Rigidity' under grant no. SFB 1442, project-ID 427320536.

\section{Recollections}

Following directly from the definition of the three kinds of mapping spaces we can describe their simplices explicitly as follows: An $n$-simplex in $\Hom_{\cC}(x,y)$ consists of a map $\Delta^n\times \Delta^1 \to \cC$, which on $\Delta^n\times \{0\}$ is the constant map to $x$, and on $\Delta^n\times \{1\}$ the constant map to $y$. An $n$-simplex in the left mapping space $\HomL_{\cC}(x,y)$ consists of an $(n+1)$-simplex $\Delta^{n+1}\to \cC$ which sends the initial vertex $\{0\}$ to $x$, and the face $d_0\Delta^{n+1}$ opposite it via the constant map to $y$. Analogously, an $n$-simplex in $\HomR_{\cC}(x,y)$ consists of an $(n+1)$-simplex $\Delta^{n+1}\to \cC$ taking the terminal vertex $\{n+1\}$ to $y$ and the face $d_{n+1}\Delta^{n+1}$ opposite it to $x$ via the constant map. 

The comparison maps among the three types of mapping spaces are then induced by certain maps 
\[
\Delta^{n+1}/d_0\Delta^{n+1}\longleftarrow \Delta^n\times \Delta^1\amalg_{\Delta^n\times \{0,1\}} \{0,1\} \longrightarrow \Delta^{n+1}/d_{n+1}\Delta^{n+1}.
\]

The other ingredient into the statements, the coherent nerve $\Nc$, is defined as $\Sing_\fC \colon \Cat^\sSet \rightarrow \sSet$, via a certain cosimplicial simplicially enriched category $\fC \colon \bbDelta \rightarrow \Cat^\sSet$. $\fC(\Delta^n)$ has objects $\{0,\dots,n\}$ and morphisms from $i$ to $j$ are given as the nerve of the poset $\{T \subseteq [i,j] \mid i,j \in T\}$ ordered by inclusion, and composition is induced by union of subsets. Note that this nerve is isomorphic to $(\Delta^1)^{\times j-i-1}$. It is a theorem of Cordier and Porter \cite{CP} that $\Nc(\C)$ is a quasi-category if all mapping complexes in $\C$ are Kan. As usual we will also denote the colimit extension $|\cdot|_\fC \colon \sSet \rightarrow \Cat^\sSet$ simply by $\fC$. The mapping complexes in the categories $\fC(X)$ are most easily described by necklaces in $X$, as expounded in \cite{DS}. While one can also use this description to obtain a proof of our main result, this route is combinatorially involved and we will rather exploit a simplified description (valid only for very particular $X$), that we explain in the next section.\\

Now to start the analysis of $\Hom_{\Nc(\C)}(x,y)$ for Kan enriched $\C$, note that an $n$-simplex in it
corresponds to a map $\Delta^n\times \Delta^1 \amalg_{\Delta^n\times\{0,1\}} \{0,1\} \to \Nc(\C)$ which sends the two $0$-simplices to $x$ and $y$, respectively. Let us write $S\Delta^n$ for the domain of this map, since it is precisely the usual (unreduced) suspension; more generally we shall write $SK$ for $K \times \Delta^1\amalg_{K\times\{0,1\}} \{0,1\}$. By adjunction such a map is the same as a functor of simplicially enriched categories, $\fC[S\Delta^n]\to \C$ sending the two objects $0,1$ to $x,y$.

As we will soon see, in the source category, the endomorphisms of $0$ and $1$ are trivial (i.e. consist of just the identity regarded as a discrete simplicial set), and the simplicial set of maps $1\to 0$ is empty. The only interesting part of $\fC[S\Delta^n]$ is therefore the simplicial set $\big(\fC[S\Delta^n]\big)(0,1)$, which one readily checks to depend functorially on $n \in \Delta$.

\begin{definition}
Denote the cosimplicial simplicial set $n \mapsto \big(\fC[S\Delta^n]\big)(0,1)$ by $\Sg \colon \bbDelta \rightarrow \sSet$.\footnote{The abbreviation $\Sg$ is meant to indicate the word `Wurst', since in particular $\Sg_3$ looks like one.}
\end{definition}

\begin{figure}
\centering
\begin{tikzpicture}
\begin{scope}[shift={(0,0)}] 
  \path[draw, mid arrow] (0,0) -- (0,1);
  \path[draw, mid arrow] (0,0) -- (1,0);
  \path[draw, mid arrow] (0,0) -- (0,-1);
  \path[draw, mid arrow] (0,0) -- (-1,0);
  \path[draw, color=red, mid arrow] (0,1) -- (1,1);
  \path[draw, mid arrow] (1,0) -- (1,1);
  \path[draw, mid arrow] (1,0) -- (1,-1);
  \path[draw, mid arrow] (0,-1) -- (1,-1);
  \path[draw, mid arrow] (0,-1) -- (-1,-1);
  \path[draw, color=red, mid arrow] (-1,0) -- (-1,-1);

  \begin{scope}[shift={(-0.2,0.2)}] 
  \node at (-0.3,0.3) {$d_1$};
  \path[draw, mid arrow] (0,0) -- (0,1);
  \path[draw, mid arrow] (0,0) -- (-1,0);
  \end{scope}

  \begin{scope}[shift={(0.25,0)}] 
  \node at (1.3,0) {$d_2$};
  \path[draw, mid arrow] (1,0) -- (1,1);
  \path[draw, mid arrow] (1,0) -- (1,-1);
  \end{scope}

  \begin{scope}[shift={(0,-0.25)}] 
  \node at (0,-1.3) {$d_0$};
  \path[draw, mid arrow] (0,-1) -- (1,-1);
  \path[draw, mid arrow] (0,-1) -- (-1,-1);
  \end{scope}
\end{scope}

\begin{scope}[shift={(6.5,0)}] 

  \begin{scope}[shift={(0,-1.3)}] 
  \node at (1.3,-1.2) {$d_0$};
  \path[draw, dashed, mid arrow] (0,-1) -- (-1,-1);
  \path[draw, dashed, mid arrow] (0,-1) -- (1,-1);
  \path[draw, dashed, mid arrow] (0,-1) -- (-0.7,-1.5);
  \path[draw, dashed, mid arrow] (0,-1) -- (0.7,-0.5);
  \path[draw, dashed, color=red, mid arrow] (0.7,-0.5) -- (1.7,-0.5);
  \path[draw, dashed, color=red, mid arrow] (-1,-1) -- (-1.7,-1.5);
  \path[draw, mid arrow] (-0.7,-1.5) -- (-1.7,-1.5);
  \path[draw, mid arrow] (-0.7,-1.5) -- (0.3,-1.5);
  \path[draw, mid arrow] (1,-1) -- (0.3,-1.5);
  \path[draw, mid arrow] (1,-1) -- (1.7,-0.5);
  \end{scope}

  \begin{scope}[shift={(-2.5,0)}] 
  \node at (-0.3,0.3) {$d_1$};
  \path[draw, mid arrow] (0,0) -- (0,1);
  \path[draw, mid arrow] (0,0) -- (-1,0);
  \path[draw, mid arrow] (0,0) -- (1,0);
  \path[draw, mid arrow] (0,0) -- (-0.7,-0.5);
  \path[draw, color=red, mid arrow] (0,1) -- (1,1);
  \path[draw, mid arrow] (1,0) -- (1,1);
  \path[draw, mid arrow] (-0.7,-0.5) -- (-1.7,-0.5);
  \path[draw, color=red, mid arrow] (-1,0) -- (-1.7,-0.5);
  \path[draw, mid arrow] (-0.7,-0.5) -- (0.3,-0.5);
  \path[draw, mid arrow] (1,0) -- (0.3,-0.5);
  \end{scope}

  \begin{scope}[shift={(0,2.5)}] 
  \node at (-0.3,0.3) {$d_2$};
  \path[draw, mid arrow] (0,0) -- (0,1);
  \path[draw, mid arrow] (0,0) -- (-1,0);
  \path[draw, color=red, mid arrow] (0,1) -- (0.7,1.5);
  \path[draw, dashed, mid arrow] (0,0) -- (0.7,0.5);
  \path[draw, dashed, mid arrow] (0.7,0.5) -- (0.7,1.5);
  \path[draw, dashed, mid arrow] (0.7,0.5) -- (0.7,-0.5);
  \path[draw, dashed, mid arrow] (0,0) -- (0,-1);
  \path[draw, dashed, mid arrow] (0,-1) -- (0.7,-0.5);
  \path[draw, dashed, mid arrow] (0,-1) -- (-1,-1);
  \path[draw, color=red, dashed, mid arrow] (-1,0) -- (-1,-1);
  \end{scope}

  \begin{scope}[shift={(1.9,0)}] 
  \node at (2,0.25) {$d_3$};
  \path[draw, mid arrow] (1,0) -- (1,1);
  \path[draw, mid arrow] (1,0) -- (1,-1);
  \path[draw, mid arrow] (1,0) -- (0.3,-0.5);
  \path[draw, mid arrow] (1,0) -- (1.7,0.5);
  \path[draw, color=red, mid arrow] (1,1) -- (1.7,1.5);
  \path[draw, mid arrow] (1.7,0.5) -- (1.7,1.5);
  \path[draw, mid arrow] (1.7,0.5) -- (1.7,-0.5);
  \path[draw, mid arrow] (1,-1) -- (1.7,-0.5);
  \path[draw, mid arrow] (1,-1) -- (0.3,-1.5);
  \path[draw, color=red, mid arrow] (0.3,-0.5) -- (0.3,-1.5);
  \end{scope}

  \begin{scope}
    \foreach \y in {-0.5,-0.3,-0.1,0.1,0.3,0.5,0.7,0.9,1.1,1.3}
    {
      {
          \path[fill=red, opacity=0.3] (0.7,\y) rectangle +(1,0.1);
      }
    }
    \path[draw, dashed, mid arrow] (0.7,0.5) -- (0.7,1.5);
    \path[fill=red, opacity=0.3] (-1,0) -- (-1.7,-0.5) -- (-1.7,-1.5) -- (-1,-1) -- cycle;
    \path[fill=red, opacity=0.3] (0,1) -- (1,1) -- (1.7,1.5) -- (0.7,1.5) -- cycle;
    \path[draw, mid arrow] (0,0) -- (1,0);
    \path[draw, mid arrow] (1,0) -- (1,1);
    \path[draw, mid arrow] (1,0) -- (1,-1);
    \path[draw, mid arrow] (0,0) -- (-1,0);
    \path[draw, dashed, mid arrow] (-1,0) -- (-1,-1);
    \path[draw, mid arrow] (0,0) -- (0,1);
    \path[draw, dashed, mid arrow] (0,0) -- (0,-1);
    \path[draw, dashed, mid arrow] (0,-1) -- (1,-1);
    \path[draw, dashed, mid arrow] (0,-1) -- (-1,-1);
    \path[draw, mid arrow] (0,1) -- (1,1);
    \path[draw, mid arrow] (0,0) -- (-0.7,-0.5);
    \path[draw, mid arrow] (-1,0) -- (-1.7,-0.5);
    \path[draw, mid arrow] (1,0) -- (0.3,-0.5);
    \path[draw, dashed, mid arrow] (0,-1) -- (-0.7,-1.5);
    \path[draw, dashed, mid arrow] (-1,-1) -- (-1.7,-1.5);
    \path[draw, mid arrow] (1,-1) -- (0.3,-1.5);
    \path[draw, mid arrow] (1,0) -- (1.7,0.5);
    \path[draw, mid arrow] (1,1) -- (1.7,1.5);
    \path[draw, mid arrow] (1,-1) -- (1.7,-0.5);
    \path[draw, mid arrow] (1.7,0.5) -- (1.7,-0.5);
    \path[draw, mid arrow] (1.7,0.5) -- (1.7,1.5);
    \path[draw, mid arrow] (1.7,0.5) -- (1.7,1.5);
    \path[draw, mid arrow] (0,1) -- (0.7,1.5);
    \path[draw, mid arrow] (0.7,1.5) -- (1.7,1.5);
    \path[draw, dashed, mid arrow] (0,0) -- (0.7,0.5);
    \path[draw, dashed, mid arrow] (0.7,0.5) -- (1.7,0.5);
    \path[draw, dashed, mid arrow] (0.7,0.5) -- (0.7,-0.5);
    \path[draw, dashed, mid arrow] (0,-1) -- (0.7,-0.5);
    \path[draw, dashed, mid arrow] (0.7,-0.5) -- (1.7,-0.5);
    \foreach \x in {-1.7,-1.5,-1.3,-1.1,-0.9,-0.7,-0.5,-0.3,-0.1,0.1}
    {
      {
          \path[fill=red, opacity=0.3] (\x,-0.5) rectangle +(0.1,-1);
      }
    }
    \path[draw, mid arrow] (-0.7,-0.5) -- (-1.7,-0.5);
    \path[draw, mid arrow] (-0.7,-0.5) -- (0.3,-0.5);
    \path[draw, mid arrow] (-0.7,-0.5) -- (-0.7,-1.5);
    \path[draw, mid arrow] (-1.7,-0.5) -- (-1.7,-1.5);
    \path[draw, mid arrow] (0.3,-0.5) -- (0.3,-1.5);
    \path[draw, mid arrow] (-0.7,-1.5) -- (-1.7,-1.5);
    \path[draw, mid arrow] (-0.7,-1.5) -- (0.3,-1.5);
    \node[shift={(1.3,0.2)}] at (0.85,0.75) {$Q_{0,3}$};
    \node[shift={(1.3,0.2)}] at (0.85,-0.25) {$Q_{1,2}$};
    \node[shift={(-0.3,-1.1)}] at (0.15,-0.75) {$Q_{2,1}$};
    \node[shift={(-0.3,-1.1)}] at (-0.85,-0.75) {$Q_{3,0}$};
  \end{scope}
\end{scope}
\end{tikzpicture}
\caption{$W_2$ and $W_3$ together with their faces, red areas are collapsed to points. For the labeling of the cubes, see \ref{lem:homspacebisimplex} and the discussion immediately preceding it.}
\end{figure}
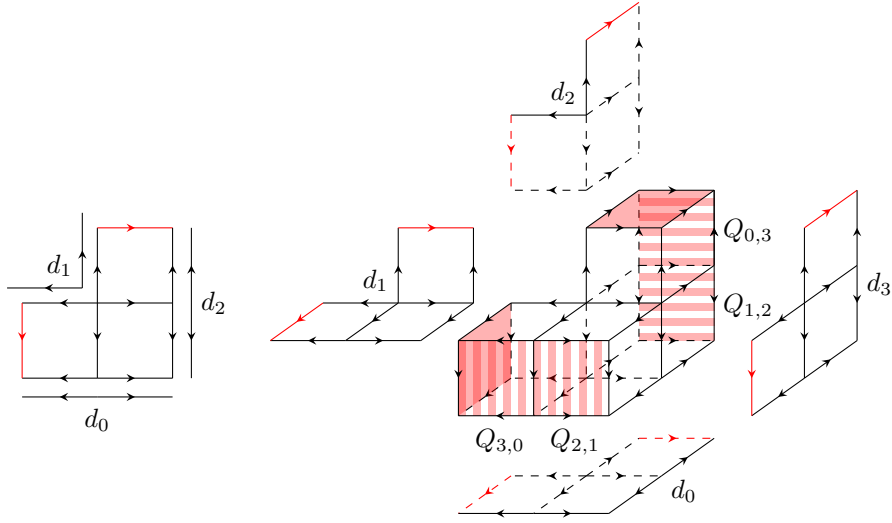

The data of a simplicially enriched functor $\fC[S\Delta^n] \rightarrow \C$ sending $0,1$ to $x,y$ is thus simply a map $\Sg_n\to \C(x,y)$. Since this correspondence is natural in $n$, we in total obtain an isomorphism
\[\Hom_{\Nc(\C)}(x,y) \cong \Sing_\Sg(\C(x,y))\]
as desired. Similar considerations yield cosimplicial objects $\QL, \QR \colon \bbDelta \rightarrow \sSet$ such that 
\[\HomL_{\Nc(\C)}(x,y) \cong \Sing_{\QL}(\C(x,y)) \quad \text{and} \quad \HomR_{\Nc(\C)}(x,y) \cong \Sing_{\QR}(\C(x,y)),\]
where $SK$ is replaced $S^L(K) = (\Delta^0 * K) \amalg_{K} \Delta^0$ and $S^R(K) = (K * \Delta^0) \amalg_{K} \Delta^0$, respectively.

The common feature of these simplicial sets, is that they are equipped with a canonical map to $\Delta^1$, whose fibres are trivial. It is this feature which allows for a fairly simple discussion of the cosimplicial objects $\Sg, \QL$ and $\QR$, as we will show in the next section.

\section{Directed two-object simplicial sets and categories}

In this section, we first explore the behavior of the adjoint pair $\fC$ and $N$ on the category of simplicially enriched categories with just two objects, and morphisms only ``in one direction'', as well as a counterpart in the world of simplicial sets. The upshot is that such simplicial sets $X$ have an associated bisimplicial set $\dec(X)$, and there is a certain bicosimplicial object $\Cb \colon \bbDelta \times \bbDelta \rightarrow \sSet$, such that 
\[\big(\fC(X)\big)(0,1) \cong |\dec(X)|_\Cb.\]
Since we will find $\Cb_{i,0} = \QL_i$ and $\Cb_{0,i} = \QR_i$ this observation offers no simplification concerning the analysis of left and right morphism complexes in coherent nerves, but with $X = S\Delta^n$ the right hand side allows for a simple analysis of $\Sg_n$, and thus middle morphism complexes, as well.

\begin{definition}
\label{def:dirsimplicial}
We call a simplicially enriched category $\C$ with two objects $0,1\in \C$ a \emph{directed two-object simplicially enriched category} if the following holds:
\begin{enumerate}
\item It has no other objects.
\item We have $\C(0,0) = \Delta^0$, $\C(1,1) = \Delta^0$.
\item We have $\C(1,0) = \emptyset$. 
\end{enumerate}
\end{definition}

Note that a directed two-object simplicially enriched category is determined by its mapping space $\C(0,1)$. In particular, there is a functor $F: \sSet \to \{0,1\}/\Cat^{\sSet}$, taking $K$ to the unique directed two-object category $\C(0,1) = K$. One immediately finds:

\begin{lemma}
\label{lem:colimpreserved}
The functor $F$ is left adjoint to the functor $\{0,1\}/\Cat^{\sSet} \rightarrow \sSet, \C \mapsto \C(0,1)$. Furthermore, $F$ is an equivalence onto the full subcategory $\Cat^\sSet_\di \subseteq \{0,1\}/\Cat^{\sSet}$ spanned by the directed two-object simplicially enriched categories.

In particular, directed two-object simplicially enriched categories are closed under colimits in $\{0,1\}/\Cat^\sSet$ and the functor $\sSet \rightarrow \Cat^{\sSet}$ obtained by composing $F$ with forgetting the marked objects preserves all colimits over connected diagrams.
\end{lemma}

For the final claim recall that the forgetful functor of any comma category preserves connected colimits.

\begin{corollary}
\label{cor:homspacepreservescolim}
The functor $\{0,1\}/\Cat^{\sSet} \to \sSet$, $\C\mapsto \C(0,1)$, preserves colimits of diagrams of directed two-object simplicially enriched categories.
\end{corollary}

We can make a definition similar to \ref{def:dirsimplicial} on the side of simplicial sets:

\begin{definition}
\label{def:dirquasi}
We call an object $K \in \partial \Delta^1/\sSet$ a \emph{directed two-object simplicial set} if it can be endowed with a map $K\to \Delta^1$ such that the diagram
\[
\begin{tikzcd}
\partial \Delta^1 \dar{\mathrm{id}}\rar & K\dar\\
\partial \Delta^1 \rar & \Delta^1
\end{tikzcd}
\]
commutes and is a pullback diagram.
\end{definition}

Note immediately, that such a map $K \rightarrow \Delta^1$ is unique if one exists, and that this is the case if and only if $K$ only has two $0$-simplices $0,1$ and every simplex $\Delta^n \rightarrow K$ can be decomposed (necessarily uniquely) as $\Delta^n = \Delta^i * \Delta^j \rightarrow K$, constant with value $0$ on $\Delta^i$ and constant with value $1$ on $\Delta^j$. 

\begin{example}
All three of $S(K), S^L(K)$ and $S^R(K)$ are naturally directed two-object simplicial sets.
\end{example}

Again the following is immediate:
\begin{lemma}
\label{lem:directedssetcolim}
The directed two-object simplicial sets form a full subcategory $\sSet_{\di} \subseteq \partial \Delta^1/\sSet$ closed under all colimits and the forgetful functor $\sSet_{\di}\to \sSet$ commutes with all connected colimits.
\end{lemma}

\begin{proposition}
\label{lem:bisimplicial}
The functor $\J: \bbDelta\times \bbDelta \to \sSet_{\di}$, 
\[
(i,j)\longmapsto (\Delta^i * \Delta^j) \amalg_{(\Delta^i\amalg \Delta^j)} (\Delta^0\amalg\Delta^0)
\]
gives rise to an equivalence
\[|\cdot|_\J \colon \ssSet \longrightarrow \sSet_\di \cocolon \Sing_\J.\]
\end{proposition}

The functor $\Sing_\J$ is the functor $\dec$, for decomposition, from the introduction to this section, and we shall use that name once the proposition is established; decoding definitions $\dec(X)$ is explicitly given by
\[(i,j) \longmapsto \Hom_{\sSet/\Delta^1}(\Delta^i * \Delta^j, X)\]
where the structure map on the left is $\Delta^i * \Delta^j \rightarrow \Delta^0 * \Delta^0 = \Delta^1$.
\begin{proof}
We will employ the following standard criterion: Suppose given a small category $A$, a cocomplete category $B$, and a functor $F: A \to B$ such that 
\begin{enumerate}
\item $F$ is fully faithful and takes values in the small objects of $B$ (i.e. those $x \in B$ for which $\Hom_B(x,-)$ commutes with colimits), and
\item the associated functor $\Sing_F$ is conservative.
\end{enumerate}
Then the adjunction
\[|\cdot|_F \colon \mathcal P(A) \longleftrightarrow B \cocolon \Sing_F\]
consists of inverse equivalences, since by (1) the unit $\mathrm{id}_{\mathcal P(A)} \Rightarrow \Sing_F \circ |\cdot|_F$ is a transformation between colimit preserving functors, that is an equivalence of representables, thus an equivalence itself, and by (2) the triangle equality then gives the same for the counit. 

These assumptions are satisfied for the functor $J: \bbDelta\times\bbDelta\to \sSet_{\di}$: To see that it is fully faithful observe first that
\[\Hom_{\sSet_\di}(\J_{i,j},\J_{k,l}) \cong \Hom_{\Delta^i \amalg \Delta^j/\sSet}(\Delta^i * \Delta^j, \J_{i,j})\]
But then since $\Delta^i * \Delta^j$ is small in $\sSet$, this receives a surjection from the subset of $\Hom_\sSet(\Delta^i * \Delta^j, \Delta^k * \Delta^l)$ taking $\Delta^i$ to $\Delta^k$ and $\Delta^j$ to $\Delta^l$. But by definition of $\J_{k,l}$ this projection identifies two maps if their entire images lie in either $\Delta^k$ or $\Delta^l$, which is impossible, so it is in fact an isomorphism. Since maps between simplices are determined by their effect on vertices we find the described subset isomorphic to $\Hom_{\bbDelta^{\times 2}}((i,j),(k,l))$ and the composite isomorphism is obviously a retract to the effect of $\J$.

To see that $\J$ takes values in the small objects note that the $n$-simplices of an object $K\in \sSet_{\di}$ come with a natural partition
\begin{equation}
\label{eq:partitionbisimplices}
K_n = \{0\}\amalg \{1\} \amalg \coprod_{i+j=n-1} \Hom_{\sSet_{\di}}(\J_{i,j}, K)\tag{$\ast$}
\end{equation}
witnessing the fact that every $n$-simplex $\sigma: \Delta^n \to K_n$ comes with a partition of its vertices into $\sigma^{-1}(0), \sigma^{-1}(1) \subseteq [n]$ and factors canonically over the respective quotient of $\Delta^n$. The formula for colimits in $\sSet_\di \subseteq \{0,1\}/\sSet$ established above then immediately implies that $\J_{i,j}$ is indeed small, and it also clear from \eqref{eq:partitionbisimplices} that $\Sing_\J$ is conservative. 
\end{proof}

Let us denote the exterior product $\sSet \times \sSet \rightarrow \ssSet$ by $\boxtimes$.

\begin{example}
\label{ex:cutjoin}
It is easy to understand the images of the relevant simplicial sets under this equivalence:
\begin{enumerate}
\item There are canonical natural isomorphisms
\[\dec(S^L(\Delta^n)) \cong \Delta^n \boxtimes \Delta^0 \quad \text{and} \quad \dec(S^R(\Delta^n)) \cong \Delta^0 \boxtimes \Delta^n\] 
since a map $\Delta^i * \Delta^j \rightarrow S^L(\Delta^{n})$ over $\Delta^1$ is the same as a map $\Delta^i * \Delta^j \rightarrow \Delta^{n+1}$ taking $\Delta^i$ to the front face and $\Delta^j$ to the vertex $n+1$.

Since the functor $S^L \colon \sSet \rightarrow \sSet_\di$ is easily checked to preserve colimits, we deduce that generally $\dec(S^LX) \simeq X \boxtimes \Delta^0$ and similarly for $S^R$.
\item Similarly, if we denote by $\Cut^n$ denote the bisimplicial set whose $i,j$-simplices are given by
\[
\Cut^n_{i,j} = \Hom_\sSet(\Delta^i * \Delta^j, \Delta^n)
\]
we find 
\[\dec(S\Delta^n) \cong \Cut^n\]
naturally in $n \in \bbDelta$ and thus generally
\[\dec(SX) \cong |X|_{\Cut}\]
by regarding $\Cut$ as a functor $\bbDelta \rightarrow \ssSet$; since simplices are small in $\sSet$ we find $|X|_{\Cut}$ given in bidegree $(i,j)$ by $\Hom_{\sSet}(\Delta^i * \Delta^j, X)$.
\item Under these identifications the natural maps 
\[S^L(K) \longleftarrow SK \longrightarrow S^R(K)\]
giving rise to the comparisons of mapping spaces correspond to the maps
\[\Delta^n \boxtimes \Delta^0 \longleftarrow \Cut^n \longrightarrow \Delta^0 \boxtimes \Delta^n\]
that simply restrict a map $\Delta^i * \Delta^j \rightarrow \Delta^n$ to $\Delta^i$ and $\Delta^j$, respectively.
\end{enumerate}
\end{example}

Next, let us observe that the coherent nerve makes these two notions of two-object objects correspond:

\begin{lemma}
\label{lem:nerveadjointpreservestwoobject}
The adjunction $\fC \colon \sSet \longleftrightarrow \Cat^\sSet \cocolon \Nc$ induces an adjunction 
\[\fC \colon \sSet_\di \longleftrightarrow \Cat^\sSet_\di \cocolon \Nc.\]
\end{lemma}
\begin{proof}
$\fC$ sends the discrete simplicial set $\{0,1\}$ to the discrete simplicially-enriched category $\{0,1\}$, and $\Nc$ does the opposite. Thus, we obtain an adjunction between slice categories ${\{0,1\}/}\sSet \leftrightarrow {\{0,1\}/}\Cat^{\sSet}$, that we claim restricts to the one from the statement. Given on the one hand an object $K\in \sSet_{\di}$, it is immediate for example from the necklace description of mapping complexes in \cite[4.4 Corollary]{DS} that $\fC(K) \in \Cat^\sSet_\di$; this also follows from \ref{lem:homspacebisimplex} below (which is independent of the present discussion): By \ref{lem:bisimplicial} $K$ is the colimit of the objects $\J_{i,j} \in \sSet_{\di}$ over the tautological diagram indexed by the category of (bi)simplices of $\dec(K)$. By adding the initial object $\{0,1\}$ of $\sSet_{\di}$ to that diagram, we get a colimit diagram which is connected and hence preserved by the forgetful functor to $\sSet$ by \ref{lem:directedssetcolim}. By \ref{lem:homspacebisimplex} $\fC$ takes all the terms of this diagram to directed two-object simplicially enriched categories, so since $\fC: \sSet\to \Cat^{\sSet}$ preserves colimits, and the forgetful functor $\{0,1\}/\Cat^{\sSet}\to \Cat^\sSet$ detects connected colimits, we see that $\fC$ takes $\sSet_{\di}$ to $\Cat^{\sSet}_{\di}$.

The analogous statement for $\Nc$ follows by direct inspection: if $\C$ is a directed two-object simplicially enriched category, then $\Nc(\C)$ has two $0$-simplices $0,1$. A functor from $\fC[\Delta^n]\to \C$ necessarily is an increasing map $[n]\to [1]$ on objects, and thus induces a map $\Nc(\C)\to \Delta^1$. Functors $\fC[\Delta^n]\to \C$ sending all objects to $0$ or all objects to $1$ are constant, so $\partial \Delta^1 \times_{\Delta^1} \Nc(\C) = \partial\Delta^1$ as required.
\end{proof}

Under the equivalences $\sSet_\di \simeq \ssSet$ and $\Cat^\sSet_\di \simeq \sSet$ from \ref{lem:colimpreserved} and \ref{lem:bisimplicial} the functor $\fC \colon \sSet_\di \rightarrow \Cat^\sSet_\di$ therefore corresponds to a colimit preserving functor $\ssSet \rightarrow \sSet$. Unwinding definitions its governing bicosimplicial object is
\[(i,j) \longmapsto \big(\fC(\J_{i,j})\big)(0,1),\]
which is the object $\Cb$ from the introduction to this section. 

We immediately note that by \ref{ex:cutjoin} we have $\Sg_n \cong |\Cut^n|_\Cb$, whereas 
\[\QL_n \cong |\Delta^n \boxtimes \Delta^0|_\Cb = \Cb_{n,0} \quad \text{and} \quad \QR_n \cong |\Delta^0 \boxtimes \Delta^n|_\Cb = \Cb_{0,n},\]
naturally in $n \in \bbDelta$.
Now the computation of the left morphism complexes in $\Nc(\C)$ is achieved by understanding $\QL$, see \cite[Section 2.2.2]{HTT}, and the analysis immediately generalises to give:

\begin{lemma}
\label{lem:homspacebisimplex}
For $i,j\geq 0$, we have that $\fC[\J_{i,j}]$ is the two-object directed simplicially enriched category with mapping space
\[
\Cb_{i,j} = \mathrm{N}(\{S \subseteq [i] * [j] \mid [i] \cap S, [j] \cap S \neq \emptyset\})/\sim
\]
where two chains of subsets $S_0\subseteq \ldots \subseteq S_n$ and $S'_0\subseteq \ldots \subseteq S'_n$ are identified if there exist $i_0\in [i]$ and $j_0\in [j]$ with $\{i_0,j_0\}\subseteq S_0, S_0'$ and $S_k\cap [i_0,j_0] = S_k'\cap [i_0,j_0]$ for all $k$.

The identification is natural for $(i,j) \in \bbDelta^{\times 2}$ for the obvious bicosimplicial structure on the right hand side.
\end{lemma}

Decoding, one finds $\Cb_{i,j}$ a quotient of an $i+j$-dimensional cube with a bunch of faces collapsed to lower-dimensional cubes.

\begin{proof}
Let us abbreviate
\[\mathrm{N}(\{S \subseteq [i] * [j] \mid [i] \cap S, [j] \cap S \neq \emptyset\})/\sim\ \cong \mathrm{N}(\{S \subseteq [i+1+j] \mid 0,i+1+j \in S\})/\sim\]
to $\mathrm{N}_{i,j}$. We thus have to produce a natural bijection, for any simplicially enriched category $\C$, between the set of maps $\J_{i,j} \rightarrow \Nc(\C)$
and the set of simplicially enriched functors $F(\mathrm{N}_{i,j}) \rightarrow \C$. 
Both sets of data decompose as disjoint unions over those maps taking the $0$-simplices/objects to $x,y \in \C$.
It is thus enough to produce, for any fixed $x,y\in \C$, a natural bijection between the set of maps $\J_{i,j} \rightarrow \Nc(\C)$ and the set of functors $F(\mathrm{N}_{i,j}) \rightarrow \C$ each sending $0,1$ to $x,y$.

Now a map $\J_{i,j} \rightarrow \Nc(\C)$ corresponds to a map $\Delta^i*\Delta^j \to \Nc(\C)$, which is constant on $\Delta^i$ and $\Delta^j$ with values $x$ and $y$, respectively. By adjunction, this corresponds to a functor $\fC[\Delta^{i+1+j}] = \fC[\Delta^{i} * \Delta^j]\to \C$ which takes the full simplicial subcategory spanned by the objects $0,\ldots,i$ to the object $x$ (and its identity morphism), and that spanned by $i+1,\ldots,i+j+1$ to $y$. We will abbreviate $\fC[\Delta^{i+1+j}]$ to $\fC$ for the remainder of this proof.

Note that a functor $\fC \rightarrow \C$ as just described is completely characterised by its restriction to $\fC(0,i+j+1) \to \C(x,y)$: On $\fC(k,l)$ it is required to be constant if $k,l\leq i$ or $k,l\geq i+1$, and compatibility with composition translates to the commutativity of the diagrams
\[
\begin{tikzcd}
\fC(l,i+j+1)\times \fC(k,l) \times \fC(0,k)\rar\dar & \fC(0,i+j+1) \dar\\
\fC(k,l) \rar & \C(x,y)
\end{tikzcd}
\]
for each $k\leq i < l$, which determine the value of the functor on $\fC(k,l)$ in terms of the lower horizontal map, since the left vertical map is surjective. What's more, any collection of maps $\fC(k,l)\to \C(x,y)$ fitting into such commutative diagrams is easily checked to determine a functor $\fC \rightarrow \C$ as desired. But essentially by definition we have a pushout diagram
\[\xymatrix{\coprod_{k\leq i < l} \fC(l,i+1+j) \times \fC(k,l) \times \fC(0,k)\ar[r]^-{\circ} \ar[d]^{\mathrm{pr}} & \fC(0,i+1+j) \ar[d]^{\mathrm{pr}}\\
             \coprod_{k\leq i < l} \fC(k,l) \ar[r]^-{\mathrm{pr}}&  \mathrm{N}_{i,j}},\]
so this data is the same as a map $\mathrm{N}_{i,j} \rightarrow \C(x,y)$. 
We leave naturality to the reader.
\end{proof}

\begin{corollary}\label{Qiscontractible}
The simplicial set $\Cb_{i,j}$ is weakly contractible for all $(i,j) \in \bbDelta^{\times 2}$.
\end{corollary}

\begin{proof}
One readily checks that the nullhomotopy
\[\mathrm{N}(\{S \subseteq [i+1+j] \mid 0,i+1+j\in S\}) * \Delta^0 \rightarrow \mathrm{N}(\{S \subseteq [i+1+j] \mid 0,i+1+j \in S\})\]
contracting the $i+j$-cube to its terminal vertex $[i+1+j]$, factors over $\Cb_{i,j}$.
\end{proof}

We also observe that $\Cb_{i,j} \cong \Cb_{j,i}$ via the unique isomorphism $\tau \colon [i] * [j] \cong [j] * [i]$, that flips $[i]$ and $[j]$ and also reverses each internally, in formulae
\[\tau(k_l) = (i-k)_r \quad \text{and} \quad \tau(k_r) = (j-k)_l,\]
where we have used the subscripts $l$ and $r$ to indicate the join factors. One easily checks that this map descends to one between $\Cb_{i,j}$ and $\Cb_{j,i}$. As far as the bicosimplicial structure is concerned denote by $X^\mathrm{fl}$ the object obtained from $X \colon \bbDelta \times \bbDelta \rightarrow B$ by precomposing with the flip of the two simplex categories and by $X^\mathrm{lrev}$ and $X^\mathrm{rrev}$ the objects obtained by precomposing with the non-trivial automorphism of $\bbDelta$ in either the left or right factor, respectively, and write $X^\mathrm{rev}$ for the composition of these two operations (and the analogous operation on cosimplicial objects). Then one readily checks:

\begin{corollary}\label{weirdequivs}
The construction above provides an isomorphism $\Cb^\mathrm{fl} \cong \Cb^\mathrm{rev}$ of bicosimplicial objects. In particular, $\QR \cong (\QL)^\mathrm{rev}$ and $W^\mathrm{rev} \cong W$.
\end{corollary}

Let us stress again, that the superscript $\mathrm{rev}$ denotes the reversal of the cosimplicial direction, not the simplicial one, for which we will use the superscript $\mathrm{op}$.

\begin{proof}
The first part is a simple check and the second part follows formally from the coend formula for realisations together with the observations that 
\begin{align*}
\big(n \mapsto (\Delta^n \boxtimes \Delta^0)^{\mathrm{fl, op}}\big) &= \big(n \mapsto \Delta^0 \boxtimes (\Delta^n)^\mathrm{op}\big) \\
&= \big(n \mapsto \Delta^0 \boxtimes \Delta^n\big)^{\mathrm{rev}}
\end{align*}
and similarly
\[\big(n \mapsto (\Cut^n)^{\mathrm{fl, op}}\big) = \big(n \mapsto \Cut^n\big)^\mathrm{rev},\]
both of which are consequences of
\[\big(n \mapsto (\Delta^n)^\mathrm{op}\big) = \big(n \mapsto \Delta^n\big)^\mathrm{rev}.\]
\end{proof}

Now for $A \colon \bbDelta \rightarrow \sSet$ and $X \in \sSet$ one generally has
\[\Sing_{A^\mathrm{rev}}(X) = \Sing_A(X)^\mathrm{op}\]
naturally in both variables, so we find:

\begin{corollary}\label{weirdequivsII}
The construction above induces isomorphisms
\[\HomR_{\Nc(\C)}(x,y) \cong \HomL_{\Nc(\C)}(x,y)^\mathrm{op} \quad \text{and} \quad \Hom_{\Nc(\C)}(x,y) \cong \Hom_{\Nc(\C)}(x,y)^\mathrm{op}.\]
\end{corollary}

In particular, the middle mapping space in a coherent nerve really is symmetric. We also invite the reader to check that the diagram
\[\xymatrix{\HomR_{\Nc(\C)}(x,y) \ar@{<->}[r] \ar[d] & \HomL_{\Nc(\C)}(x,y)^\mathrm{op} \ar[d] \\
\Hom_{\Nc(\C)}(x,y) \ar@{<->}[r] & \Hom_{\Nc(\C)}(x,y)^\mathrm{op},}\]
where the vertical maps are the canonical inclusions of the one-sided into the middle morphism complexes, commutes. \\

\section{Proof of the main result}

We start out by producing the comparison map $\sigma \colon \Sg \rightarrow \Delta$, which in fact factors as $\Sg \rightarrow \QL \rightarrow \Delta$.\footnote{We think of the map $\Sg \rightarrow \QL$ as squeezing the Wurst onto its Zipfel.} To this end just observe that a natural transformation $|\Cut^n|_\Cb = \Sg_n \rightarrow \Delta^n$ corresponds to a natural map $\Cut^n \rightarrow \Sing_\Cb(\Delta^n)$ of bisimplicial sets. But in bidegree $(i,j)$ these are given by
\[\Hom_\sSet(\Delta^i * \Delta^j, \Delta^n) \longrightarrow \Hom_\sSet(\Cb_{i,j},\Delta^n)\]
so it suffices to give maps $\Cb_{i,j} \rightarrow \Delta^i * \Delta^j$ natural in $(i,j) \in \Delta^{\times 2}$. Since the target is a full simplex, such maps can be specified on $0$-simplices, where we declare that $S \subseteq [i]*[j]$ should go to $\mathrm{max}(S \cap [i])$, considered as a $0$-simplex of $\Delta^i * \Delta^j$. The check that this is really well-defined and natural we leave to the reader, but let us note, that it takes values in $\Delta^i \subseteq \Delta^i * \Delta^j$, and therefore factors over $\QL$, compare \ref{ex:cutjoin} (3).

Composing with the equivalences from \ref{weirdequivs} yields an analogous map $\Sg \rightarrow \QR \rightarrow \Delta^\mathrm{rev} = \Delta^\mathrm{op}$ that takes $S \subseteq [i] * [j]$ to $\mathrm{min}(S \cap [j])$, but we shall not make use of it.

To obtain our main result we now want to apply Reedy's lemma (see \ref{reedy} below) to the transformation $\sigma \colon \Sg \rightarrow \Delta$ just constructed. We are thus required to check that $\Sg_n$ is weakly contractible for each $n \in \bbDelta$ and that $\Sg$ is Reedy cofibrant, i.e. that $|\partial \Delta^n|_\Sg \rightarrow |\Delta^n|_\Sg = \Sg_n$ is always a cofibration. While both statements can be checked directly, we will deduce them by further application of Reedy's lemma to the bicosimplicial simplicial set $\Cb$ using the equivalences $\Sg_n = |\Cut^n|_\Cb$.

Let us first state Reedy's lemma, which we will need for both simplicial and bisimplicial sets (it generally holds over any Reedy category, which both $\bbDelta$ and $\bbDelta^{\times 2}$ are examples of). To give a uniform statement, let us agree that by a simplex we shall mean one of the $\Delta^i$ in the case of $\sSet$ and one of the $\Delta^i \boxtimes \Delta^j$ in the case of $\ssSet$. Similarly, by the boundary of a simplex we shall mean either $\partial \Delta^i$ or 
\[\partial(\Delta^i \boxtimes \Delta^j) = (\partial(\Delta^{i-1}) \boxtimes \Delta^{j}) \cup (\Delta^i \boxtimes \partial\Delta^{j-1}),\]
as appropriate. A (bi)cosimplicial simplicial set $X$ is called Reedy cofibrant if $|\cdot|_X$ takes these boundary inclusions to cofibrations (i.e. degreewise injections). Here is the version of Reedy's result we shall need:

\begin{proposition}[Reedy's lemma]\label{reedy}
Let $X$ be a Reedy cofibrant (bi)cosimplicial simplicial set, whose structure maps are weak homotopy equivalences. Then $|\cdot|_X$ preserves degreewise injections and weak homotopy equivalences, and in the cosimplicial case $\Sing_X$ preserves homotopy equivalences between Kan complexes. 

If furthermore, $\eta \colon X \rightarrow Y$ is a pointwise weak equivalence between two such (bi)cosimplicial simplicial sets, then 
\[\eta_* \colon |S|_X \rightarrow |S|_Y\]
is a weak homotopy equivalence for every (bi)simplicial set $S$, and in the cosimplicial case for any Kan complex $T$ the map
\[\eta^* \colon \Sing_Y(T) \longrightarrow \Sing_X(T)\]
is a homotopy equivalence of Kan complexes.
\end{proposition} 

\begin{proof}[Proof sketch:]
The statement about $\eta_*$ is a simple skeletal induction using the glueing lemma, and that $|\cdot|_X$ preserves cofibrations is similarly easy (and both statements only use Reedy cofibrancy of $X$ and $Y$). The second part implies that $\Sing_X$ preserves trivial fibrations. Next, one shows that $|\cdot|_X$ preserves weak homotopy equivalences, for example by considering the composite
\[\operatorname{const} X_0 \times \Delta \longrightarrow X \times \Delta \xrightarrow{\mathrm{pr}} X,\]
which gives the claim by an application of the first part, since 
\[S \longmapsto |S|_{\operatorname{const} X_0 \times \Delta} = X_0 \times S\] evidently preserves weak homotopy equivalences. Investing the fact that cofibrations that are also weak equivalences are in fact anodyne, then shows that $|\cdot|_X$ also preserves anodynes and it follows that $\Sing_X$ preserves Kan fibrations, so in particular $\Sing_X(T)$ is Kan again. Using this knowledge one easily checks that the adjunction isomorphism
\[\Hom_\sSet(|S|_X,T) \cong \Hom_\sSet(S,\Sing_X(T))\]
is compatible with the homotopy relation. This immediately implies that $\Sing_X$ preserves homotopy equivalences between Kan complexes. Finally, the diagram
\[\xymatrix{[S,\Sing_Y(T)] \ar@{<->}[d] \ar[r]^{\eta^*} & [S,\Sing_X(T)] \ar@{<->}[d] \\
            [|S|_Y,T] \ar[r]^{\eta_*} & [|S|_X,T]}\]
commutes for every $S \in \sSet$ and by the previous parts has three maps bijective, hence also the fourth.
\end{proof}

\begin{remark} 
\leavevmode
\begin{enumerate}
\item The only step in this result requiring non-trivial input (namely that trivial cofibrations are anodyne) is that $\Sing_X(T)$ and $\Sing_Y(T)$ are indeed Kan complexes (which, as far as we know, is a necessary step for proving that they are even weakly equivalent). In the case we want to apply the second part of the lemma, i.e. for $X = \Sg$, $Y = \Delta$ and $T = \C(x,y)$ for some Kan enriched category $\C$, this can be avoided: It is trivial for $\Sing_Y(T)$ and follows for $\Sing_X(T)$, since by design $\Sing_\Sg(\C(x,y)) = \Hom_{\Nc(\C)}(x,y)$ and morphism complexes in quasi-categories are Kan as a straight forward consequence of Joyal's lifting theorem \cite{Jo}.
\item Alternatively, by inspection of the proof above, one only needs to know that $|\cdot|_\Sg$ preserves anodynes to conclude that $\Hom_{\Nc(\C)}(x,y)$ is Kan without investing either result. A saturation argument reduces this to the case of horn inclusions, and here an explicit verification should be possible, but we shall refrain from attempting it here.
\item The statements about $\Sing_X$ are restricted to the cosimplicial case (as opposed to the bicosimplicial one) to avoid a discussion of various possible choices of fibrations and cofibrations in $\ssSet$, as it is not required for our applications. By \cite[Section IV.3.2]{GJ} the argument above works equally well for example upon replacing the words `anodyne' and `Kan complex' by `trivial Reedy cofibration' and `Reedy fibrant bisimplicial set', respectively.
\end{enumerate}
\end{remark}

We will now set out to show that $\Cb$ satisfies the assumptions of Reedy's lemma. We already showed in \ref{Qiscontractible} that $\Cb$ is termwise contractible, so it remains to check Reedy cofibrancy. 

We will employ the following criterion:
\begin{lemma}
Let $X$ by a (bi)cosimplicial simplicial set, such that for any pair of codimension $1$ faces $K,K'$ of a simplex $L$ the diagram
\[
\begin{tikzcd}
{|}K\cap K'{|}_X \rar\dar & {|}K{|}_X \dar\\
{|}K'{|}_X \rar & {|}L{|}_X
\end{tikzcd}
\]
is cartesian. Then $X$ is Reedy cofibrant.
\end{lemma}

The point of the lemma is that all four terms in the square are actually simplices of varying dimensions (before realisation) and therefore the diagram itself is given by evaluating $X$ on a diagram, with no colimits occuring in the computation of the four corners. 

Note also, that the case $K=K'$ in the condition above boils down to precisely the statement that the boundary maps of $X$ are cofibrations (which is automatic as the degeneracies give left inverses).

\begin{proof}
It is easy to give a hands-on argument; here is an abstract one:

The bisimplicial set $\partial L \subseteq L$ is the levelwise image of the tautological map $F = \coprod_{K} K\to L$, where the index runs through all codimension $1$-faces of $L$. In the category of sets, images of an arbitrary morphism $X\to Y$ can be computed by $X\to \Coeq(X\times_Y X\tworightarrows X)\to Y$, in the sense that the composition is the original map, the first map surjective, and the second one injective. Consequentially, this also holds in any $\Set$-valued functor category, like $\sSet$. It follows that $\partial L \rightarrow L$ agrees with the canonical map $\Coeq(F\times_{L} F \tworightarrows F) \rightarrow L$. We can therefore compute
\begin{align*}
|\partial L|_X &= \Coeq(|F\times_{L} F|_X \tworightarrows |F|_X) \\
               &= \Coeq(\coprod_{K,K'}|K\times_{L} K'|_X \tworightarrows |F|_X) \\
               &= \Coeq(\coprod_{K,K'}|K|_X\times_{|L|_X} |K'|_X \tworightarrows |F|_X) \\
               &= \Coeq(|F|_X\times_{|L|_X} |F|_X \tworightarrows |F|_X),
\end{align*}
where we have used that pullbacks commute with colimits in the second and last steps, and the assumption in the third one. By the considerations above, this last expression agrees with the image of $|\partial L|_X$ in $|L|_X$, which forces the induced map to be injective in each degree.
\end{proof}

\begin{corollary}\label{Qcof}
The bicosimplicial simplicial set $\Cb$ is Reedy cofibrant.
\end{corollary}

This statement is the only piece of combinatorics we did not manage to avoid; alas the work has to be somewhere.

\begin{proof}
Recall then that 
\[\Cb_{i,j} = \mathrm{N}(\{S \subseteq [i] * [j] \mid [i] \cap S, [j] \cap S \neq \emptyset\})/\sim
\]
where two chains of subsets $S_0\subseteq \ldots \subseteq S_n$ and $S'_0\subseteq \ldots \subseteq S'_n$ are identified if there exist $i_0\in [i]$ and $j_0\in [j]$ with $\{i_0,j_0\}\subseteq S_0, S_0'$ and $S_k\cap [i_0,j_0] = S_k'\cap [i_0,j_0]$ for all $k$.

We wish to apply the criterion from the previous lemma, which will require a case distinction based on whether the two faces are taken in the same simplicial direction or not. Consider the first case and, say, faces in the first direction, the arguments for the second being dual. Then we need to show that for any $0 \leq l < k \leq i$ the intersection of $d_k\Cb_{i-1,j}$ and $d_l\Cb_{i-1,j}$ is no larger than $d_kd_l\Cb_{i-2,j}$. 
Denote by 
\[\pi \colon \mathrm{N}(\{S \in [i+1+j] \mid [i] \cap S, [i+1,i+1+j] \cap S\neq \emptyset \}) \longrightarrow \Cb_{i,j}\] the projection. Then one easily checks that $\pi^{-1}(d_k\Cb_{i-1,j})$ consists of all those sequences $S_0 \subseteq \dots \subseteq S_n$ for which
\[S_0 \cap [k+1,i] \neq \emptyset \quad \text{or} \quad k \notin S_n.\]
Similarly, one checks that both the preimage of the intersection and $\pi^{-1}(d_kd_l\Cb_{i-2,j})$ consist of those sequences with 
\begin{enumerate}
\item $S_0 \cap [k+1,i] \neq \emptyset$,
\item $S_0 \cap [l+1,i] \neq \emptyset$, but $k \notin S_n$, or
\item $k,l \notin S_n$.
\end{enumerate}

For the mixed case, say regarding $d_k\Cb_{i-1,j}$ and $d_l\Cb_{i,j-1}$ both the preimage of the intersection and the preimage of $d_{k,l}\Cb_{i-1,j-1}$ consist of those chains with 
\begin{enumerate}
\item $S_0 \cap [k+1,i] \neq \emptyset$ and $S_0 \cap [i+1,i+l] \neq \emptyset$,
\item $S_0 \cap [k+1,i] \neq \emptyset$ and $i+1+l \notin S_n$, 
\item $k \notin S_n$ and $S_0 \cap [i+1,i+l] \neq \emptyset$, or finally
\item $k,i+1+l \notin S_n$.
\end{enumerate}
The claim follows.
\end{proof}

As the final ingredient we need:

\begin{lemma}\label{cutcontr}
The bisimplicial sets $\Cut^n$ are weakly contractible.
\end{lemma}

\begin{proof}

Let $J: \bbDelta^{\times 2} \rightarrow \sSet$ be the functor given by $(i,j) \mapsto \Delta^i * \Delta^j$. From Reedy's lemma we obtain weak homotopy equivalences
\[|\Cut^n|_J \longleftarrow |\Cut^n|_{J \times \Delta^{\boxtimes 2}} \longrightarrow |\Cut^n|_{\Delta^{\boxtimes 2}} = \mathrm{diag}(\Cut^n),\]
where $\Delta^{\times 2}$ is the bicosimplicial simplicial set $(i,j)\mapsto \Delta^i\times \Delta^j$. By the same argument as in \ref{ex:cutjoin} we have $|\Cut^n|_J \cong \Delta^n \times \Delta^1$.
\end{proof}

\begin{proof}[Proof of the main result]
Since $\Hom_{\Nc(\C)}(x,y) \cong \Sing_\Sg(\C(x,y))$ by construction of $\Sg$ an application of Reedy's lemma to the map $\sigma \colon \Sg \rightarrow \Delta$ constructed at the beginning of this section will give the claim. We therefore need to check that $\Sg$ is Reedy cofibrant and degreewise weakly contractible. To this end, recall that $\Sg_n = |\Cut^n|_\Cb$ naturally in $n \in \bbDelta$, or in other words $|\cdot|_\Sg =  |\cdot|_\Cb \circ |\cdot|_{\Cut}$. But both of these preserve degreewise injections, $|\cdot|_{\Cut}$ by inspection, see \ref{ex:cutjoin}, and $|\cdot|_\Cb$ by \ref{Qcof}. Furthermore, by \ref{Qiscontractible} another application of Reedy's lemma gives weak homotopy equivalences
\[\Sg_n = |\Cut^n|_\Cb \longleftarrow |\Cut^n|_{\Cb \times \Delta^{\boxtimes 2}} \longrightarrow |\Cut^n|_{\Delta^{\boxtimes 2}} = \mathrm{diag}(\Cut^n)\]
the right hand term of which is contractible by the previous lemma.
\end{proof}

Note that \ref{Qcof} also implies that $\QL$ is Reedy cofibrant and thus 
we find that in fact all three maps in the diagram
\[\xymatrix{
 & \C(x,y)\ar[dl]_\sigma \ar[dr]^\sigma & \\
 \Hom^\mathrm{L}_{\Nc(\C)}(x,y)	\ar[rr] && \Hom_{\Nc(\C)}(x,y)}
	\]
consist of homotopy equivalences. As mentioned it is a less elementary result of Joyal that the lower horizontal map is an equivalence for any quasicategory in place of the homotopy coherent nerve, see for example \cite[Section 4.2.1]{HTT} for an exposition.

\end{document}